\documentclass[12pt]{article}
\usepackage{amsmath}
\usepackage{amsfonts}
\usepackage{amssymb}
\usepackage{graphicx}
\usepackage{amsthm}
\usepackage{enumerate}
\usepackage{tikz}
\usepackage{url,xargs}
\usepackage{hyperref}
\usepackage{textcomp}
\usepackage{soul}

\begin{document}

\title{On cubic bridgeless graphs whose edge-set cannot be covered by
  four perfect matchings}

\author{L. Esperet \footnote{Laboratoire G-SCOP (Grenoble-INP, CNRS),
    Grenoble, France. Partially supported by the French \emph{Agence
      Nationale de la Recherche} under reference
    \textsc{anr-10-jcjc-0204-01}.} \and
  G. Mazzuoccolo \footnote{Laboratoire G-SCOP (Grenoble-INP, CNRS),
    Grenoble, France. Research supported by a fellowship from the
    European Project ``INdAM fellowships in mathematics and/or
    applications for experienced researchers cofunded by Marie Curie
    actions'' e-mail: {\tt mazzuoccolo@unimore.it}}}

\maketitle

\newtheorem{theorem}{Theorem}
\newtheorem{lemma}[theorem]{Lemma}
\newtheorem{definition}[theorem]{Definition}
\newtheorem{example}[theorem]{Example}
\newtheorem{corollary}[theorem]{Corollary}
\newtheorem{conjecture}[theorem]{Conjecture}
\newtheorem{observation}[theorem]{Observation}

\newcommand{\Ha}{\mathring{H}}
\def\scc{\mathop{\mathrm{scc}}\nolimits}

\begin{abstract} \noindent 
The problem of establishing the number of perfect matchings necessary
to cover the edge-set of a cubic bridgeless graph is strictly related
to a famous conjecture of Berge and Fulkerson.  In this paper we prove
that deciding whether this number is at most 4 for a given cubic
bridgeless graph is NP-complete. We also construct an infinite family
$\cal F$ of snarks (cyclically 4-edge-connected cubic graphs of girth
at least five and chromatic index four) whose edge-set cannot be
covered by 4 perfect matchings. Only two such graphs were known. 

It turns out that the family $\cal F$ also has interesting properties
with respect to the shortest cycle cover problem.  The shortest cycle
cover of any cubic bridgeless graph with $m$ edges has length at least
$\tfrac43m$, and we show that this inequality is strict for graphs of
$\cal F$. We also construct the first known snark with no cycle cover
of length less than $\tfrac43m+2$.
\end{abstract}

\noindent \textit{Keywords: Berge-Fulkerson conjecture, perfect
  matchings, shortest cycle cover, cubic graphs, snarks.}

\section{Introduction}\label{sec:intro}

Throughout this paper, a graph $G$ always means a simple connected
finite graph (without loops and parallel edges). A \emph{perfect
  matching} of $G$ is a $1$-regular spanning subgraph of $G$. In this
context, a \textit{cover} of $G$ is a set of perfect matchings of
$G$ such that each edge of $G$ belongs to at least one of the perfect
matchings. Following the terminology introduced in \cite{BonCar}, the
\textit{excessive index} of $G$, denoted by $\chi'_e(G)$, is the least
integer $k$ such that the edge-set of $G$ can be covered by $k$
perfect matchings.

A famous conjecture of Berge and Fulkerson~\cite{Ful} states that the
edge-set of every cubic bridgeless graph can be covered by 6 perfect
matchings, such that each edge is covered precisely twice. The second
author recently proved that this conjecture is equivalent to another
conjecture of Berge stating that every cubic bridgeless graph has
excessive index at most five~\cite{Maz}.

Note that a cubic bridgeless graph has excessive index 3 if and only
if it is 3-edge-colorable, and deciding the latter is
NP-complete. H\"agglund~\cite[Problem 3]{Hag} asked if it is possible
to give a characterization of all cubic graphs with excessive index
5. In Section~\ref{sec:np}, we prove that the structure of cubic
bridgeless graphs with excessive index at least five is far from
trivial. More precisely, we show that deciding whether a cubic
bridgeless graph has excessive index at most four (resp. equal to
four) is NP-complete.

\medskip

The gadgets used in the proof of NP-completeness have many
2-edge-cuts, so our first result does not say much about
3-edge-connected cubic graphs. A \emph{snark} is a non
$3$-edge-colorable cubic graph with girth (length of a shortest cycle)
at least five that is cyclically $4$-edge connected. A question raised
by Fouquet and Vanherpe~\cite{FouVan} is whether the Petersen graph is
the unique snark with excessive index at least five. This question was
answered by the negative by H\"agglund using a computer
program~\cite{Hag}. He proved that the smallest snark distinct from
the Petersen graph having excessive index at least five is a graph
$\Ha$ on 34 vertices (see Figure \ref{fig:Hagg}). In
Section~\ref{sec:cons} we show that the graph found by H\"agglund is a
special member of an infinite family $\cal F$ of snarks with excessive
index precisely five.

\medskip

In Section~\ref{sec:scc}, we study shortest cycle covers of the graphs
from our family $\cal F$. A \emph{cycle cover} of a graph $G$ is a
covering of the edge-set of $G$ by cycles (connected subgraphs with
all degrees even), such that each edge is in at least one cycle. The
\emph{length} of a cycle cover is the sum of the number of edges in
each cycle. The \emph{Shortest Cycle Cover Conjecture} of Alon and
Tarsi~\cite{AT85} states that every bridgeless graph $G$ has a cycle
cover of length at most $\tfrac75 |E(G)|$. This conjecture implies a
famous conjecture due to Seymour~\cite{Sey79} and
Szekeres~\cite{Sze73}, the \emph{Cycle Double Cover Conjecture}, which
asserts that every bridgeless graph has a cycle cover such that every
edge is covered precisely twice (see~\cite{JT92}). It turns out that
the Cycle Double Cover Conjecture also has interesting connections
with the excessive index of snarks. Indeed, it was proved
independently by Steffen~\cite{Ste} and Hou, Lai, and
Zhang~\cite{HouLaiZha} that it is enough to prove the Cycle Double
Cover conjecture for snarks with excessive index at least five.

\smallskip

The best known upper bound on the length of a cycle cover of a
bridgeless graph $G$, $\tfrac53 |E(G)|$, was obtained by Alon and
Tarsi~\cite{AT85} and Bermond, Jackson, and Jaeger~\cite{BJJ83}. For
cubic bridgeless graphs there is a trivial lower bound of
$\tfrac{4}{3} |E(G)|$, which is tight for 3-edge-colorable cubic
graphs.  Jackson~\cite{Jac94} proved an upper bound of $\tfrac{64}{39}
|E(G)|$, and Fan~\cite{Fan94} improved it to $\tfrac{44}{27}
|E(G)|$. The best known upper bound in the cubic case, $\tfrac{34}{21}
|E(G)|$, was obtained by Kaiser, Kr\'al', Lidick\'y, Nejedl\'y, and \v
S\'amal~\cite{KKLNS10} in 2010.

Brinkmann, Goedgebeur, H\"agglund, and Markstr\"om~\cite{BriGoeHAgMar}
proved using a computer program that the only snarks $G$ on $m$ edges
and at most 36 vertices having no cycle cover of length
$\tfrac{4}{3}m$ are the Petersen graph and the 34-vertex graph $\Ha$
mentioned above. Moreover, these two graphs have a cycle cover of
length $\tfrac{4}{3}m+1$. They also conjectured that every snark $G$
has a cycle cover of size at most $(\tfrac{4}{3}+o(1)) |E(G)|$.

In Section~\ref{sec:scc}, we show that all the graphs $G$ in our
infinite family $\cal F$ have shortest cycle cover of length more than
$\tfrac{4}{3} |E(G)|$. We also find the first known snark with no
cycle cover of length less than $\tfrac{4}{3} |E(G)|+2$ (it has 106
vertices).

\paragraph{Notation}

Let $X,Y$ be subsets of the vertex-set of $G$. We denote by $\partial
X$ the set of edges with precisely one end-vertex in $X$ and by
$\partial(X,Y)$ the set of edges with one end-vertex in $X$ and the
other in $Y$. If $X,Y$ form a partition of the vertex-set of $G$, the
set $\partial(X,Y)$ is called an \emph{edge-cut}, or a
\emph{$k$-edge-cut} if $k$ is the cardinality of $\partial(X,Y)$. It
is well-known that for any $k$-edge-cut $\partial(X,Y)$ and any
perfect matching $M$ of a cubic graph $G$, the numbers $|X|$, $|Y|$,
$k$, and $|M\cap \partial(X,Y)|$ have the same parity. This will be
used implicitly several times in the remaining of the paper.

\section{NP-completeness}\label{sec:np}

We denote by $T$ the \emph{Tietze graph}, that is the graph obtained
from the Petersen graph by replacing one vertex by a triangle (see
Figure~\ref{fig:tietze}).

\begin{figure}[ht]
\centering
\includegraphics[width=12cm]{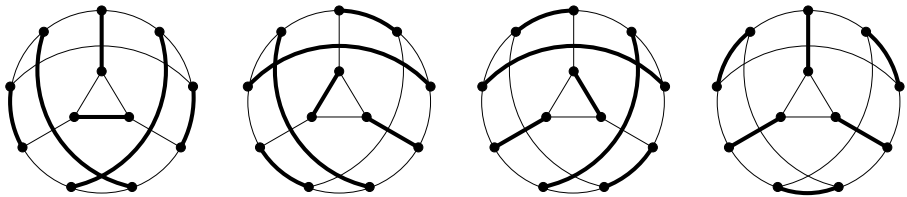}
\caption{Four perfect matchings covering the edge-set of the Tietze
  graph.}\label{fig:tietze}
\end{figure}

It is not difficult to check that the Petersen graph has excessive
index $5$ (see also~\cite{FouVan}), while the Tietze graph has
excessive index $4$ (see Figure~\ref{fig:tietze} for a cover
consisting of four perfect matchings). The next lemma describes how
four perfect matchings covering the edge-set of $T$ intersect the
unique triangle of $T$.

\begin{lemma}\label{lem:tietze}
For any cover $\mathcal M$ of $T$ with $|{\mathcal M}|=4$, each edge
of the unique triangle of $T$ is in precisely one perfect matching of
$\mathcal M$.
\end{lemma}

\begin{proof}
By parity, each perfect matching of $\mathcal M$ contains one or three
edges incident to the unique triangle of $T$.  If a perfect matching
of $\mathcal M$ contains the three edges incident to $T$, then each of the
three other perfect matchings of $\mathcal M$ must contain a distinct
edge of $T$, hence the lemma holds. On the other hand, suppose for the
sake of contradiction that each perfect matching of $\mathcal M$
contains precisely one edge incident to the triangle. Then by
contracting the triangle we obtain four perfect matchings covering the
Petersen graph, a contradiction.
\end{proof}

Before we prove the main result of this section, we need to describe
an operation on cubic graphs. Given two cubic graphs $G$ and $H$ and
two edges $xy$ in $G$ and $uv$ in $H$, the
\emph{gluing}\footnote{This operation is sometimes called the
  \emph{2-cut-connection}.} of $(G,x,y)$ and $(H,u,v)$ is the graph
obtained from $G$ and $H$ by removing edges $xy$ and $uv$, and
connecting $x$ and $u$ by an edge, and $y$ and $v$ by an edge. Note
that if $G$ and $H$ are cubic and bridgeless, the resulting graph is
also cubic and bridgeless.

When the order of each pair $(x,y)$ and $(u,v)$ is not important, we
simply say that \emph{we glue the edge $uv$ of $H$ on the edge $xy$ of $G$}. 
We obtain either the gluing of $(G,x,y)$ and $(H,u,v)$
or the gluing of $(G,y,x)$ and $(H,u,v)$, but which one does not
matter.

\smallskip

We now study the complexity of bounding the excessive index of a cubic
bridgeless graph. A cubic bridgeless graph has excessive index three
if and only if it is 3-edge-colorable, and determining the latter is a
well-known NP-complete problem (see \cite{Hol}). We now prove that
determining whether the excessive index is at most 4 (or equal to 4)
is also hard.

\begin{theorem}\label{th:NPcomplete}
Determining whether a cubic bridgeless graph $G$ satisfies
$\chi'_e(G) \leq 4$ (resp. $\chi'_e(G)=4$) is an NP-complete problem.
\end{theorem}
\begin{proof}
The proof proceeds by reduction to the {\sc 3-edge-colorability} of
cubic bridgeless graphs, which is NP-complete~\cite{Hol}. Our problem
is certainly in NP, since a cover of $G$ consisting of four perfect
matchings gives a certificate that can be checked in polynomial time.

Now, let $G$ be a cubic bridgeless graph, and let $n$ be its number of
vertices. Let $G'$ be the graph obtained from $G$ by replacing each
vertex by a triangle. Note that $G'$ is a cubic bridgeless graph with
$3n$ vertices and $n$ (vertex-disjoint) triangles. The $3n$ edges of
$G'$ contained in the $n$ triangles are called the \emph{new edges},
while the other are called the \emph{original edges}. Let $uv$ be an
edge of the unique triangle of the Tietze graph $T$, and let $H$ be
the graph obtained from $G'$ and $3n$ copies of $T$ indexed by the
$3n$ new edges of $G'$, by gluing each new edge $e$ of $G$ on the
edge $uv$ of the copy $T_e$ of $T$. Note that $H$ is a cubic
bridgeless graph with $39n$ vertices. We first remark that if $H$ can
be covered by $k$ perfect matchings, then $T$ can also be covered by
$k$ perfect matchings since for any copy $T_e$, any perfect matching
of $H$ contains 0 or 2 edges connecting $T_e$ to $G$ (and thus
corresponds to a perfect matching of $T$). It follows that $H$ is not
3-edge-colorable.

\smallskip

We now prove that $G$ is 3-edge-colorable if and only if $\chi'_e(H)=
4$ (which is equivalent to $\chi'_e(H)\leq 4$ by the previous remark).

Suppose first that $G$ is 3-edge-colorable, and consider three perfect
matchings $M_1,M_2,M_3$ covering the edges of $G$. Each perfect
matching $M_i$, $i \in \{1,2,3\}$, can be naturally extended to a
perfect matching $M_i'$ of $G'$ such that $\{M_1',M_2',M_3'\}$ covers
the edges of $G'$. These three perfect matchings can be combined with
the three leftmost perfect matchings of Figure~\ref{fig:tietze}, for
each copy $T_e$ glued with $G'$, to obtain three perfect matchings of
$H$ covering all the edges of $H$ except three in each copy of $T$. We
add a fourth perfect matching of $H$, obtained by combining the
perfect matching of $G'$ consisting of the original edges of $G$, with
the rightmost perfect matching of Figure~\ref{fig:tietze} in each copy
of $T$. Since all the edges of $H$ are covered by these four perfect
matchings of $H$, we have $\chi'_e(H)\leq 4$, and since $H$ is not
3-edge-colorable, $\chi'_e(H)= 4$.

Suppose now that $\chi'_e(H)= 4$, and let $M_1,M_2,M_3,M_4$ be perfect
matchings covering $H$. By Lemma~\ref{lem:tietze}, these perfect
matchings correspond to four perfect matchings $M'_1,M'_2,M'_3,M'_4$
of $G'$ covering the edges of $G'$, and such that each new edge of
$G'$ is covered once. Since every vertex of $G'$ is incident to two
new edges and one original edge, it follows that each original edge is
covered twice. For all $i \in \{1,2,3\}$, color the original edges with color $i$
if they are covered by $M'_4$ and $M'_i$, or by $M'_j$ and $M'_k$ with
$\{j,k\}=\{1,2,3\}\setminus \{i\}$. Assume that some original edge
$xy$ is covered both by $M'_4$ and $M'_1$. Then the two new edges
incident to $x$ are covered by $M'_2$ and $M'_3$, respectively, and
the last edge of the new triangle containing $x$ is covered by one of
$M'_4$ and $M'_1$. It follows that neither of the two edges incident
to this triangle and distinct from $xy$ is covered by $M'_4$ and
$M'_1$, or by $M'_2$ and $M'_3$. By symmetry, it follows that each
color class corresponds to a perfect matching of $G$, and so $G$ is
3-edge-colorable.
\end{proof}

\section{An infinite family of snarks with excessive index 5}\label{sec:cons}

Recall that a \emph{snark} is a cubic graph that is cyclically
4-edge-connected, has girth at least five, and is not
3-edge-colorable. In this section we show how to construct a snark $G$
with excessive index at least 5 from three snarks $G_0,G_1,G_2$ each
having excessive index at least 5. Taking $G_0=G_1=G_2$ to be the
Petersen graph, we obtain the graph $\Ha$ found by H\"agglund using a
computer program~\cite{Hag}, and for which no combinatorial proof
showing that its excessive index is 5 has been known up to now.  Our
proof holds for any graph obtained using this construction, thus we
exhibit an infinite family of snarks with excessive index 5. This
answers the question of Fouquet and Vanherpe~\cite{FouVan} about the
existence of snarks distinct from the Petersen graph having excessive
index 5 in a very strong sense.

\medskip

\noindent\textbf{The windmill construction}\footnote{This construction should not be confused with \emph{windmill graphs}, graphs obtained from disjoint cliques by adding a universal vertex.}\hspace{6pt}
 For $i \in \{0,1,2\}$, consider a snark $G_i$ with an edge
 $x_iy_i$. Let $x_i^0$ and $x_i^1$ (resp. $y_i^0$ and $y_i^1$) be the
 neighbors of $x_i$ (resp. $y_i$) in $G_i$. For $i\in \{0,1,2\}$, let
 $H_i$ be the graph obtained from $G_i$ by removing vertices $x_i$ and
 $y_i$. We construct a new graph $G$ from the disjoint union of $H_0$,
 $H_1$, $H_2$ and a new vertex $u$ as follows. For $i\in \{0,1,2\}$,
 we introduce a set $A_i=\{a_i,b_i,c_i\}$ of vertices such that $a_i$
 is adjacent to $x_{i+1}^0$ and $y_{i-1}^0$, $b_i$ is adjacent to
 $x_{i+1}^1$ and $y_{i-1}^1$, and $c_i$ is adjacent to $a_i$, $b_i$
 and $u$ (here and in the following all indices $i$ are taken modulo
 $3$).

\smallskip

The windmill construction is depicted in Figure~\ref{fig:mainconstruction}.

\begin{figure}[ht]
\centering
\includegraphics[width=10cm]{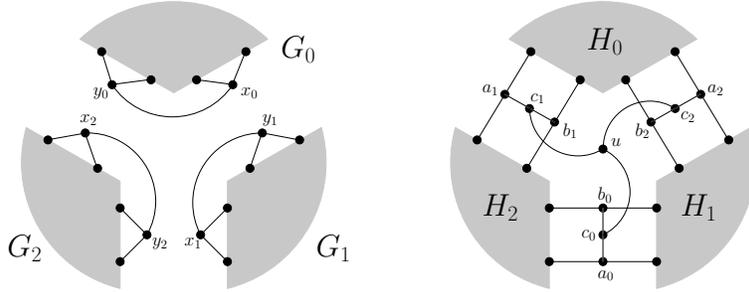}
\caption{The windmill construction of $G$ (right) from $G_0$, $G_1$, and
  $G_2$.}\label{fig:mainconstruction}
\end{figure}

\begin{lemma}\label{lem:snark}
If $G_0,G_1,G_2$ are snarks, any graph $G$ obtained from $G_0,G_1,G_2$
by the windmill construction is cubic, cyclically 4-edge-connected, and
has girth at least 5.
\end{lemma}

\begin{proof}
It is trivial to check that $G$ is cubic, bridgeless, and has girth at
least five (we do not introduce cycles of length three or four, and
the distance between any two vertices of any $H_i$ is at least as
large in $G$ as in $G_i$). Note that for any $i\in \{0,1,2\}$, the graph $G
\setminus (V(H_i)\cup V(H_{i+1}) \cup A_{i-1})$ is a subdivision of
$G_{i-1}$. It follows that for any cyclic $k$-edge-cut in $G$, we can find
a non-trivial edge-cut of cardinality at most $k$ in some $G_i$. If
$k\leq 3$ this non-trivial edge-cut is cyclic. Since all $G_i$'s are
cyclically 4-edge-connected, $G$ is also cyclically 4-edge-connected.
\end{proof}

Before we prove that $G$ has excessive index at least five, we need a
couple of definitions. By a slight abuse of notation, we denote the
set of edges of $G$ with exactly one end-vertex in $V(H_i)$ by
$\partial H_i$ (instead of $\partial (V(H_i))$). In what follows it
will be useful to consider how each perfect matching $M$ of $G$ intersects
$\partial H_i$. Since $|\partial H_i|=4$, by parity we have that
$|M\cap \partial H_i|$ is even. Observe that $|M\cap \partial H_i|\ne
4$, since otherwise $u$ would be adjacent to both $c_{i-1}$ and
$c_{i+1}$ in $M$. If $|M\cap \partial H_i|=0$, we say that $M$ is of
\textit{type $0$} on $H_i$. If $|M\cap \partial H_i|=2$ we consider
two cases: we say that $M$ is of \textit{type $1$ } on $H_i$ if
$|M\cap \partial(H_i,A_{i-1})|=|M\cap \partial(H_i,A_{i+1})|=1$, while
$M$ is of \textit{type $2$} on $H_i$ otherwise (in this case one of
the two sets of edges $M\cap \partial(H_i,A_{i-1})$, $M\cap
\partial(H_i,A_{i+1})$ has cardinality $2$ and the other is empty). A
perfect matching $M$ of $G$ of type $2$ on $H_0$, type $1$ on $H_1$
and type $0$ on $H_2$ is depicted in Figure~\ref{fig:matching_type}
(left).

\begin{figure}[ht]
\centering
\includegraphics[width=10cm]{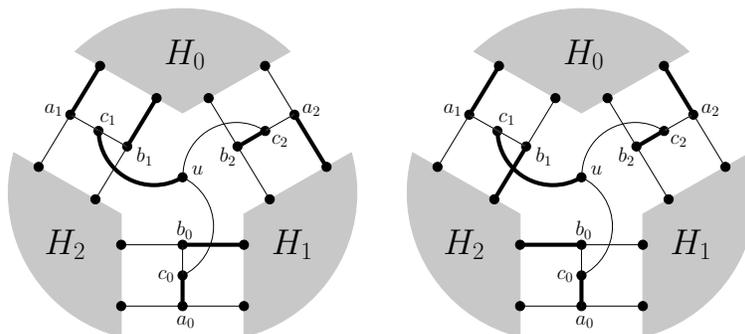}
\caption{The two kinds of perfect matching of $G$ (up to
  symmetry).}\label{fig:matching_type}
\end{figure}

Observe first that each perfect matching of type 0 on $H_i$
corresponds to a perfect matching of $G_i$ containing $x_iy_i$, while
each perfect matching of type 1 on $H_i$ corresponds to a perfect
matching of $G_i$ avoiding $x_iy_i$. This observation has the
following immediate consequence:

\begin{lemma}\label{lem:cov}
Assume that for some $i$ and $k$, $G$ can be covered by $k$ perfect matchings,
each of type 0 or 1 (and not all of type 1) on $H_i$. Then $G_i$ can
be covered by $k$ perfect matchings.
\end{lemma}

Note that a perfect matching $M$ of $G$ contains two edges of
$\partial(H_{i-1}\cup H_{i+1},A_{i})$ if $uc_i$ is an edge of $M$, and
$M$ contains one edge of $\partial(H_{i-1}\cup H_{i+1},A_{i})$
otherwise. It follows that up to symmetry (including a possible
permutation of the $H_i$'s), there are only two kinds of perfect
matchings of $G$: perfect matching of types 0, 1, and 2 (see
Figure~\ref{fig:matching_type}, left), and perfect matching of types
0, 1, and 1 (see Figure~\ref{fig:matching_type}, right). In
particular, we have the following property:

\begin{observation}\label{obs:types1}
Every perfect matching of $G$ is of type 2 on at
most one $H_i$, type 0 on exactly one $H_i$, and type 1 on at least
one $H_i$.
\end{observation}

We are now ready to prove our main theorem.

\begin{theorem}\label{th:main}
Let $G_0,G_1,G_2$ be snarks such that $\chi'_e(G_i)\geq 5$ for
$i\in \{0,1,2\}$. Then any graph $G$ obtained from $G_0,G_1,G_2$ by the
windmill construction is a snark with $\chi'_e(G)\geq 5$. 
\end{theorem}

\begin{proof}  
The graph $G$ is cubic, cyclically 4-edge-connected, and has girth at
least 5, as proved in Lemma~\ref{lem:snark}. We now prove that $G$ has
excessive index at least five (which implies that $G$ is not
3-edge-colorable, and thus a snark).  For the sake of contradiction,
let $\mathcal M = \{M_1,M_2,M_3,M_4\}$ be a cover of $G$, where $M_1$
might be identical to $M_2$ (in which case $G$ is 3-edge-colorable).

By Observation~\ref{obs:types1}, each perfect matching $M_i$ is of
type $0$ on one $H_j$, so we can assume without loss of generality
that $M_1$ and $M_2$ are of type $0$ on $H_0$. By Lemma~\ref{lem:cov},
at least one of $M_3,M_4$ is of type 2 on $H_0$, since otherwise $G_0$
would have excessive index at most four. If $M_3,M_4$ do not have the
same type on $H_0$, then some edge of $\partial H_0$ is not covered, so
it follows that both $M_3,M_4$ are of type 2 on $H_0$. More precisely,
we can assume by symmetry that $|\partial(H_0,A_1) \cap M_3|=2$ and
$|\partial(H_0,A_2) \cap M_4|=2$, {\it i.e.} $uc_1$ belongs to $M_3$ and
$uc_2$ belongs to $M_4$. This implies that $M_3$ is of type $1$ on
$H_1$ and of type $0$ on $H_2$, while $M_4$ is of type $0$ on $H_1$
and of type $1$ on $H_2$.

Since the edge $uc_0$ is not covered by $M_3 \cup M_4$, we can assume
without loss of generality that $M_1$ contains $uc_0$. This implies
that $M_1$ is of type $1$ on both $H_1$ and $H_2$. Combining this with
the conclusion of the previous paragraph, it follows by
Lemma~\ref{lem:cov} that $M_2$ is of type 2 on $H_1$ and $H_2$, which
contradicts Observation~\ref{obs:types1}.
\end{proof}

Let $\mathcal F$ be the family of all graphs inductively defined as
being either the Petersen graph, or a graph obtained from three graphs
$G_0,G_1,G_2$ of $\mathcal F$ by the windmill
construction. Theorem~\ref{th:main} shows that $\mathcal F$ is an
infinite family of snarks with excessive index at least $5$.

\medskip

We now prove that every graph from $\mathcal F$ has excessive index
precisely five. We will indeed prove the stronger statement that every
graph of $\mathcal F$ satisfies the Berge-Fulkerson conjecture,
i.e. has six perfect matchings covering each edge precisely
twice. Such a cover of a cubic bridgeless graph $G$, called a
\emph{Berge-Fulkerson cover}, can be equivalently defined as a coloring
$C$ of the edges of $G$ with subsets of size two of
$\{1,2,3,4,5,6\}$ such that for each triple of edges
$e_1,e_2,e_3$ sharing a vertex, the sets $C(e_i)$ are pairwise
disjoint (equivalently, $C(e_1)\cup C(e_2) \cup
C(e_3)=\{1,2,3,4,5,6\}$).

\begin{lemma}\label{lem:bf}
If $G_0,G_1,G_2$ are cubic bridgeless graphs with a Berge-Fulkerson
cover, then any graph $G$ obtained from $G_0,G_1,G_2$ by the windmill
construction has a Berge-Fulkerson cover.
\end{lemma}

\begin{proof}
Consider a Berge-Fulkerson cover $C$ of $G_0$, and let
$e_1,e_2,e_3,e_4$ be the four edges sharing a vertex with the edge
$x_0y_0$, such that $e_1,e_2$ are incident to $x_0$ and $e_3,e_4$ are
incident to $y_0$. Let $\{\alpha,\beta\}= C(e_1)$ and
$\{\gamma,\delta\}= C(e_2)$. Then at least one of $\gamma,\delta$, say
$\gamma$, is such that $C(e_3)$ and $C(e_4)$ are distinct from
$\{\alpha,\gamma\}$. It follows that, after renaming the colors of
$C$, we can assume without loss of generality that $i \in C(e_i)$ for
all $1\leq i \leq 4$, while $x_0y_0$ is colored $\{5,6\}$.

\begin{figure}[ht]
\centering
\includegraphics[width=13cm]{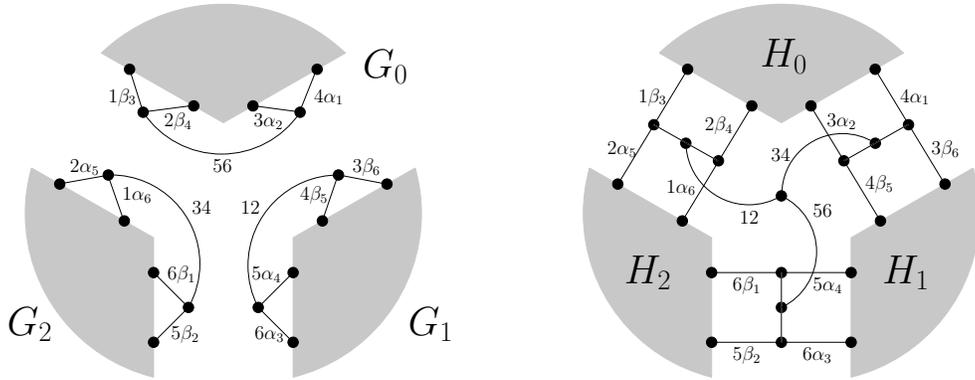}
\caption{The extension of Berge-Fulkerson covers of $G_0,G_1,G_2$ to
  $G$. Colors are denoted by $ij$ instead of $\{i,j\}$.}\label{fig:bf}
\end{figure}

By symmetry, it follows that $G_0,G_1,G_2$ have a Berge-Fulkerson
cover as depicted in Figure~\ref{fig:bf} (left), where for every
$(i,j)\in \{(1,2),(3,4),(5,6)\}$, $\{\alpha_i,\alpha_j\}=
\{\beta_i,\beta_j\}=\{i,j\}$. These covers can be extended into a
Berge-Fulkerson cover of $G$, as described in Figure~\ref{fig:bf}
(right). In order not to overload the figure, we have omitted the
colors of the edges $a_ic_i$ and $b_ic_i$, $0\leq i \leq 2$. It can be
checked that taking $C(a_0c_0)=\{\beta_1,\alpha_4\}$,
$C(b_0c_0)=\{\beta_2,\alpha_3\}$, $C(a_1c_1)=\{\beta_4,\alpha_6\}$,
$C(b_1c_1)=\{\beta_3,\alpha_5\}$, $C(a_2c_2)=\{\alpha_2,\beta_5\}$,
and $C(b_2c_2)=\{\alpha_1,\beta_6\}$ defines a Berge-Fulkerson cover
of $G$.
\end{proof}

Theorem~\ref{th:main} and Lemma~\ref{lem:bf} have the following
immediate corollary:

\begin{corollary}
The family $\mathcal F$ is an infinite family of snarks having excessive
index precisely five.
\end{corollary}

\section{Shortest cycle cover}\label{sec:scc}

The length of a shortest cycle cover of a bridgeless graph $G$ is
denoted by $\scc(G)$. Note that for any cubic graph $G$,
$\scc(G)\geq \tfrac43 |E(G)|$. The purpose of this section is to
show that for any graph $G$ in the family $\mathcal F$ constructed in
the previous section, $\scc(G)>\tfrac43 |E(G)|$. Given some
integer $k$, a \emph{$k$-parity subgraph} of a cubic graph $G$ is a
spanning subgraph of $G$ in which all vertices have degree one, except
$k$ vertices that have degree 3. We first observe the following:

\begin{observation}\label{obs:scc43}
For any cubic graph $G$ and any cycle cover $\cal C$ of $G$, $\cal C$
has length $\tfrac43 |E(G)|$ if and only if the set of edges covered
twice by $\cal C$ is a perfect matching of $G$; and $\cal C$ has
length $\tfrac43 |E(G)|+1$ if and only if the set of edges covered twice
by $\cal C$ is a 1-parity subgraph of $G$.
\end{observation}

\begin{proof}
For each vertex $v$ of $G$, let $d_{\cal C}(v)$ be the sum of the
degree of $v$ in each of the cycles of $\cal C$. Since $\cal C$ is a
cycle cover and $G$ is cubic, for each vertex $v$, $d_{\cal C}(v)\ge
3$ and $d_{\cal C}(v)$ is an even integer. In particular $d_{\cal
  C}(v)\ge 4$. It follows that $\cal C$ has length $\tfrac43 |E(G)|$
if and only if for each vertex $v$, $d_{\cal C}(v)=4$. This is
equivalent to the fact that two edges incident to $v$ are covered once
by $\cal C$, and the third edge is covered twice by $\cal C$. Hence,
$\cal C$ has length $\tfrac43 |E(G)|$ if and only if the set of edges
covered twice by $\cal C$ is a perfect matching of $G$.

Similarly, $\cal C$ has length $\tfrac43 |E(G)|+1$ if and only if
there is a vertex $u$ with $d_{\cal C}(u)=6$ and for each vertex $v\ne
u$, $d_{\cal C}(v)=4$. Note that as above, if $d_{\cal C}(v)=4$, then two edges
incident to $v$ are covered once by $\cal C$, and the third edge is
covered twice by $\cal C$. In particular, no edge of $G$ is
covered more than twice by $\cal C$. Since $d_{\cal C}(u)=6$, the only
possibility is that each of the three edges incident to $u$ is
covered twice by $\cal C$. Equivalently, the set of edges covered
twice by $\cal C$ is a 1-parity subgraph of $G$.
\end{proof}

Let $G_0,G_1,G_2$ be snarks and let $G$ be a snark obtained from
$G_0,G_1,G_2$ by the windmill construction.  Recall that the
\emph{type} of a perfect matching of $G$ on $H_i$ was defined in the
previous section. We define the type of a 1-parity subgraph $P$ of $G$
on $H_i$ similarly: If $|P\cap \partial H_i|=j$ with $j\in\{0,4\}$, we
say that $P$ is of \textit{type $j$} on $H_i$ (as it was observed in
the previous section, type $4$ cannot occur when $P$ is a perfect
matching). If $|P\cap \partial H_i|=2$ we consider two cases: we say
that $P$ is of \textit{type $1$} on $H_i$ if $|P\cap
\partial(H_i,A_{i-1})|=|P\cap \partial(H_i,A_{i+1})|=1$, while $P$ is
of \textit{type $2$} on $H_i$ otherwise (in this case one of the two
sets of edges $P\cap \partial(H_i,A_{i-1})$, $P\cap
\partial(H_i,A_{i+1})$ has cardinality $2$ and the other is empty).

\smallskip

Let ${\cal C}$ be a cycle cover of $G$. For any edge $e$, we denote by
${\cal C}(e)$ the set of cycles of ${\cal C}$ containing the edge $e$.
In the next two lemmas, we fix $i,j\in \{0,1,2\}$ with $i\ne j$,
and call $e_j$ and $f_j$ the edges connecting $H_i$ to $a_j$ and $b_j$
respectively.

\begin{lemma}\label{lem:sccgi43}
Assume $G$ has a cycle cover $\cal C$ of length $\tfrac43 |E(G)|$. Let
$M$ be the perfect matching of $G$ consisting of the edges covered
twice by $\cal C$. If there are distinct cycles
$x,y\in \cal C$, such that either (i) $M$ is of type 0 on $H_i$ and
$\{{\cal C}(e_{j}),{\cal C}(f_{j})\}=\{\{x\},\{y\}\}$, or (ii) $M$ is of
type 1 on $H_i$ and $\{{\cal C}(e_{j}),{\cal
  C}(f_{j})\}=\{\{x\},\{x,y\}\}$, then $G_i$ has a cycle cover ${\cal
  C}_i$ of length $\tfrac43 |E(G_i)|$.
\end{lemma} 

\begin{proof}
In both cases we obtain a cycle cover of $G_i$ such that the set of
edges covered twice is a perfect matching of $G_i$ (see
Figure~\ref{fig:sccext}, where cycles of $G_i$ are represented by
dashed or dotted lines for more clarity). By
Observation~\ref{obs:scc43} this cycle cover has length $\tfrac43
|E(G_i)|$. Note that in case (ii) (see Figure~\ref{fig:sccext},
right), $x$ and $z$ might be the same cycle.
\end{proof}

\begin{figure}[ht]
\centering
\includegraphics[width=10cm]{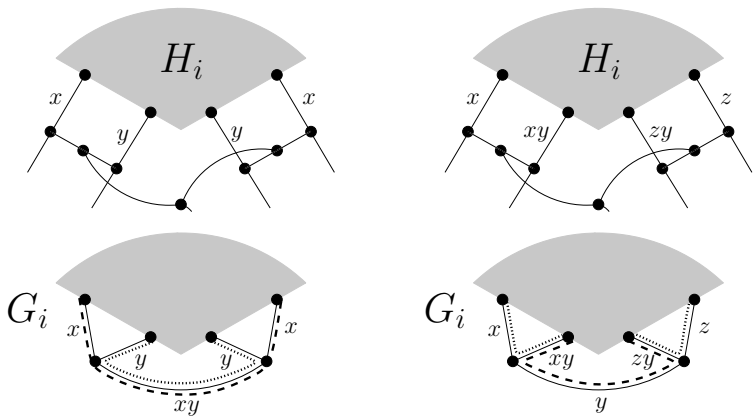}
\caption{Constructing cycle covers of $G_i$ from cycle covers of
  $G$.}\label{fig:sccext}
\end{figure}

Given a cycle cover $\cal C$ of $G$ of length $\tfrac43 |E(G)|+1$, the
unique vertex of degree 3 in the 1-parity subgraph of $G$ associated
to $\cal C$ will be denoted by $t_{\cal C}$ (or simply $t$ if the cycle
cover is implicit). The same proof as above, applied this time to
1-parity subgraphs, gives the following slightly different lemma.

\begin{lemma}\label{lem:sccgi431}
Assume $G$ has a cycle cover $\cal C$ of length $\tfrac43
|E(G)|+1$. Let $P$ be the 1-parity subgraph of $G$ consisting of the
edges covered twice by $\cal C$. If there are distinct cycles $x,y\in
\cal C$, such that either (i) $P$ is of type 0 on $H_i$ and $\{{\cal
  C}(e_{j}),{\cal C}(f_{j})\}=\{\{x\},\{y\}\}$, or (ii) $P$ is of type
1 on $H_i$ and $\{{\cal C}(e_{j}),{\cal
  C}(f_{j})\}=\{\{x\},\{x,y\}\}$, then $G_i$ has a cycle cover ${\cal
  C}_i$ of length at most $\tfrac43 |E(G_i)|+1$. Moreover, if $t_{\cal
  C}$ does not lie in $H_i$ or its neighborhood in $G$, then ${\cal
  C}_i$ has length at most $\tfrac43 |E(G_i)|$.
\end{lemma} 

\begin{figure}[ht]
\centering
\includegraphics[width=6cm]{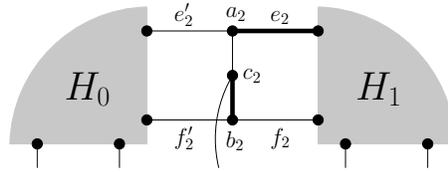}
\caption{Notation in the proof of Theorem \ref{th:scc431}.} \label{fig:th12}
\end{figure}

We now use Lemma~\ref{lem:sccgi43} to prove the main result of this
section.

\begin{theorem}\label{th:scc431}
Let $G_0,G_1,G_2$ be snarks such that $\scc(G_i)> \frac43 |E(G_i)|$
for $i\in \{0,1,2\}$. Then any graph $G$ obtained from $G_0,G_1,G_2$
by the windmill construction is a
snark with $\scc(G)> \frac43 |E(G)| $.
\end{theorem}
\begin{proof}
Assume that $G$ has a cycle cover $\cal C$ of length $\frac43 |E(G)|$,
and let $M$ be the perfect matching consisting of the edges covered
twice by $\cal C$. By Observation \ref{obs:types1}, we can assume that
$M$ is of type $0$ on $H_0$ and type $1$ on $H_1$. Let $e_2,f_2$ be
the edges connecting $H_1$ to $a_2,b_2$ and let $e_2',f_2'$ be the
edges connecting $H_0$ to $a_2,b_2$ (see Figure~\ref{fig:th12}). By
symmetry we can assume that $e_2$ and $b_2c_2$ are in $M$. Let $x,y
\in \cal C$ be the two cycles covering $e_2$. Then one of them, say
$x$, also covers $b_2c_2$, and therefore covers either $f_2$ or
$f_2'$. If $x$ covers $f_2'$ we obtain a contradiction with
Lemma~\ref{lem:sccgi43}, case (i), and if $x$ covers $f_2$ we obtain a
contradiction with Lemma~\ref{lem:sccgi43}, case (ii).
\end{proof}

Recall that beside the Petersen graph, only one snark $G$ with
$\scc(G)>\frac43 |E(G)| $ was known~\cite{BriGoeHAgMar} (and the
proof of it was computer-assisted). Theorem~\ref{th:scc431} has the
following immediate corollary:

\begin{corollary}
The family $\mathcal F$ is an infinite family of snarks such that for
any $G\in {\cal F}$, $\scc(G)>\frac43 |E(G)| $.
\end{corollary}

\begin{figure}[ht]
\centering
\includegraphics[width=6cm]{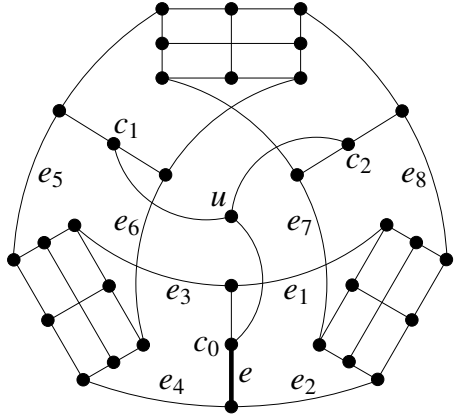}
\caption{The graph $\Ha$ in Lemma~\ref{lem:noci}. The bold edge $e$
  will be used in the proof of Theorem \ref{th:scc432}.} \label{fig:Hagg}
\end{figure}

We now prove that in a cycle cover $\cal C$ of the graph $\Ha$ of
Figure~\ref{fig:Hagg} of length $\tfrac43 |E(\Ha)|+1$, some vertices
cannot play the role of $t_{\cal C}$ (the only vertex of degree 3 in
the 1-parity subgraph of $\Ha$ associated to $\cal C$).

\begin{lemma}\label{lem:noci}
Let $G_0,G_1,G_2$ be three copies of the Petersen graph. Then the
graph $\Ha$ obtained from $G_0,G_1,G_2$ by the windmill construction
is a snark such that:

1. $\scc(\Ha)=\frac43 |E(\Ha)| + 1$ \smallskip

2. For every cycle cover $\cal C$ of $\Ha$ of minimum length, the
vertex $t_{\cal C}$ cannot be one of the vertices $c_i$ (see
Figure~\ref{fig:Hagg}).
\end{lemma}
\begin{proof}
The first part of the lemma was proved in \cite{BriGoeHAgMar}. Assume
now that the second part fails for some cycle cover $\cal C$ of length
$\frac43 |E(\Ha)| + 1$, and let $P$ be the 1-parity subgraph associated
to $\cal C$. By symmetry, we suppose $t_{\cal C}=c_0$. It implies
that all three edges incident to $c_0$ are in $P$. As a consequence,
each edge $e_i$, $1\leq i \leq 4$, is covered only once by $\cal C$
(see Figure~\ref{fig:Hagg}). It follows also that one edge incident to
$c_1$ and distinct from $c_1u$ is covered twice, and by parity this
implies that $e_5$ and $e_6$ are covered once. By a symmetric
argument, $e_7$ and $e_8$ are covered once, therefore $P$ is of type 0
on $H_1$ and $H_2$. The edges incident to $c_0$ are covered by three
different cycles, so at least one of the pairs $e_1,e_2$ and $e_3,e_4$
is such that the two edges in the pair are not covered by the same
cycle of $\cal C$. This contradicts Lemma~\ref{lem:sccgi431}, case
(i), since the Petersen graph has no cycle cover of length 20.
\end{proof}

\begin{figure}[ht]
\centering
\includegraphics[width=10cm]{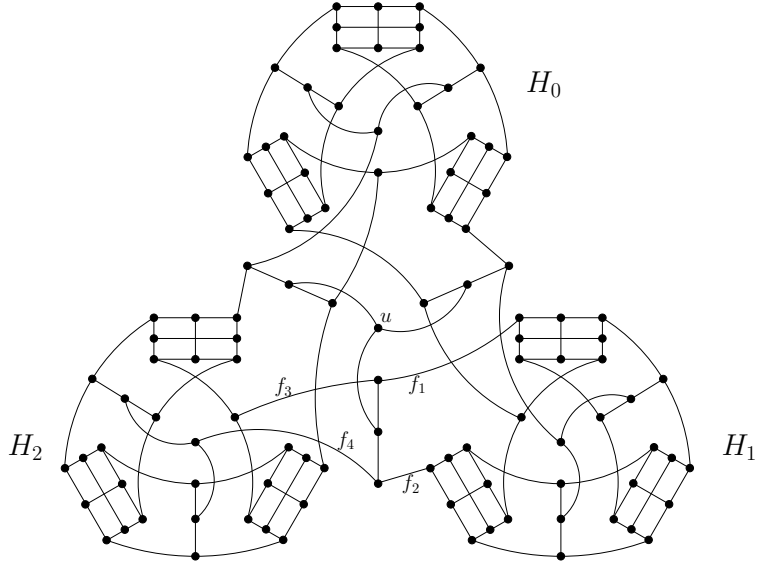}
\caption{A snark $G$ with $\scc(G)\geq \tfrac43 |E(G)| + 2
  $.}\label{fig:scc}
\end{figure}

We now use Lemma~\ref{lem:noci} to exhibit the first known snark $G$
satisfying $\scc(G)> \tfrac43 |E(G)| + 1 $.

\begin{theorem}\label{th:scc432}
Any cycle cover of the graph $G$ depicted in Figure
\ref{fig:scc} has length at least $\tfrac43 |E(G)| + 2 $.
\end{theorem}
\begin{proof}
The graph $G$ is obtained from three copies $G_0,G_1,G_2$ of the graph
$\Ha$ of Figure~\ref{fig:Hagg} by the windmill construction, where
each $H_i$ is obtained from $G_i$ by removing the endpoints of the
edge $e$ (see Figure~\ref{fig:Hagg}).  By Theorem \ref{th:scc431},
$\scc(G)>\frac43 |E(G)|$. Assume for the sake of contradiction that
$G$ has a cycle cover $\cal C$ of length $\tfrac43 |E(G)| + 1 $, and
let $P$ be the 1-parity subgraph associated to $\cal C$ and $t=t_{\cal
  C}$ be the only vertex of degree three in $P$.

We claim that if $P$ is of type 0 or 1 on $H_i$, then $t$ lies in
$H_i$. For the sake of contradiction, assume that $P$ is of type 0 or
1 on $H_i$, while $t$ lies outside of $H_i$ (note that $t$ might be
adjacent to some vertex of $H_i$). Then $P$ induces a perfect matching
of $G_i$. The graph $G_i$ itself was obtained from three copies
$G_0',G_1',G_2'$ of the Petersen graph by the windmill
construction. By Observation~\ref{obs:types1}, $P$ is of type 0 on
some $G_i'$, and type 1 on some $G_j'$. Therefore we can assume,
following the notation of Figure \ref{fig:Hagg}, that $e,e_1 \in P$
and $e_2,e_3,e_4 \not\in P$. Since the Petersen graph has no cycle
cover of length 20, we can assume by Lemma~\ref{lem:sccgi43}, case
(i), that $e_3$ and $e_4$ belong to the same cycle of $\cal C$, and by
Lemma~\ref{lem:sccgi43}, case (ii), that the cycle containing $e_2$
does not contain $e_1$. It can be checked that this is a
contradiction.

\medskip

Now, if $t$ lies inside some $H_i$, remark that $P$
satisfies the conclusion of Observation~\ref{obs:types1} (since the
proof of this observation only considers the intersection of $P$ and
the edges outside the $H_i$'s). In particular, $P$ is of type 0 on
some $H_j$, and type 1 on some $H_k$. This contradicts the claim of
the previous paragraph. Similarly, if $t$ is outside of the $H_i$'s,
but distinct from $u$ (see Figure~\ref{fig:scc}), it can be checked
that $P$ is of type 1 on some $H_i$, which again contradicts the
previous paragraph. It follows that $t$ coincides with $u$ and $P$ is
of type 2 or 4 on each $H_i$. This implies that $P$ is of type 2 on
each $H_i$, so either the two edges $f_1$,$f_2$ or the two edges
$f_3$,$f_4$ belong to $P$. Then $P$ induces a 1-parity subgraph of
$G_1$ (in the first case) or $G_2$ (in the second case), such that the
unique vertex of degree three of this 1-parity subgraph coincides with
a vertex $c_j$ of $G_i$, contradicting Lemma~\ref{lem:noci}.
\end{proof}

\section{Open problems}
H\"agglund proposed the following two problems (Problems 3 and 4 in
\cite{Hag}):
\begin{enumerate}
\item Is it possible to give a simple characterization of cubic graphs
  $G$ with $\chi'_e(G)=5$?
\item Are there any cyclically $5$-edge-connected snarks $G$ with
  excessive index at least five distinct from the Petersen graph?
\end{enumerate}
While the former problem has a negative answer by Theorem
\ref{th:NPcomplete} of the present paper (unless $P=N\! P$), the latter one is still
open, since each element of the infinite family $\mathcal F$ contains
cyclic 4-edge-cuts. The edge-connectivity also plays an important role
in the proof of Theorem~\ref{th:NPcomplete}, in particular the gadgets
we use have many $2$-edge-cuts. Hence we leave open the problem of
establishing whether it is possible to give a simple characterization
of $3$-edge-connected or cyclically $4$-edge-connected cubic graphs
with excessive index 5.

In Section \ref{sec:scc}, we have shown interesting properties of
$\mathcal F$ with respect to cycle covers. In particular, Theorem
\ref{th:scc432} proves the existence of a snark $G\in \cal F$ with no cycle
cover of length less than $\frac 43|E(G)|+2$. We believe that there
exist snarks in $\mathcal F$ for which the constant $2$ can be
replaced by an arbitrarily large number. On the other hand, recall
that Brinkmann, Goedgebeur, H\"agglund, and
Markstr\"om~\cite{BriGoeHAgMar} conjectured that every snark $G$ has a
cycle cover of size at most $(\tfrac{4}{3}+o(1)) |E(G)|$.

\end{document}